\documentclass[leqno,12pt]{amsart} 
\usepackage{amssymb}
\theoremstyle{plain} 
\newtheorem{theorem}{\indent\sc Theorem}[section]
\newtheorem{lemma}[theorem]{\indent\sc Lemma}

\newtheorem{proposition}[theorem]{\indent\sc Proposition}

\theoremstyle{definition} 
\newtheorem{definition}[theorem]{\indent\sc Definition}
\newtheorem{remark}[theorem]{\indent\sc Remark}
\newtheorem{example}[theorem]{\indent\sc Example}

\usepackage[active]{srcltx}
\usepackage{pdfsync}
\usepackage{tabularx} 
\usepackage{arydshln}

\newcommand{\CC}{\mathbb{C}}

\newcommand{\FF}{\mathbb{F}}
\newcommand{\ZZ}{\mathbb{Z}}
\newcommand{\HH}{\mathbb{H}}
\newcommand{\OO}{\mathbb{O}}
\newcommand{\RR}{\mathbb{R}}

\newcommand{\GL}{\mathrm{GL}}

\newcommand{\Sp}{\mathrm{Sp}}
\newcommand{\SO}{\mathrm{SO}}
\newcommand{\SU}{\mathrm{SU}}

\renewcommand{\hom}{\mathrm{Hom}}

\def\dim{\mathop{\hbox{\rm dim}}}
\def\tr{\mathop{\rm tr}}
  
\newcommand{\spf}{\mathfrak{sp}}
\newcommand{\der}{\mathfrak{der}}
\newcommand{\hol}{\mathfrak{hol}}
\newcommand{\inder}{\mathfrak{inder}}
\newcommand{\suf}{\mathfrak{su}}
\newcommand{\sof}{\mathfrak{so}}
\newcommand{\slf}{\mathfrak{sl}}
\newcommand{\gl}{\mathfrak{gl}}
\newcommand{\ad}{\mathop{\mathrm{ad}}}
\newcommand{\id}{\mathrm{id}}
\newcommand{\mm}{\mathfrak{m}}
\newcommand{\hh}{\mathfrak{h}}
\newcommand{\g}{\mathfrak{g}}

\title[  3-Sasakian  manifolds and symplectic triple systems]{  Holonomy and 3-Sasakian homogeneous manifolds \\versus symplectic triple systems}

\author[C.~Draper]{ 
Cristina Draper Fontanals${}^*$ 
}  

\thanks{${}^*$ Supported by the Spanish Ministerio de Econom\'{\i}a y Competitividad---Fondo Europeo de
Desarrollo Regional (FEDER) MTM2016-76327-C3-1-P, and by the Junta de Andaluc\'{\i}a grant FQM-336, with FEDER funds}

\address{%
Departamento de Matem\'{a}tica Aplicada,  Escuela de Ingenier\'\i as Industriales,\endgraf
Universidad de M\'{a}laga, 
 29071 M\'{a}laga,  
Spain
}
\email{cdf@uma.es}
\begin{document}


\maketitle 

\vspace{-15pt}\begin{center}\scriptsize{\textrm{Dedicated to the memory of Professor Thomas Friedrich}}\end{center}

\begin{abstract}
Our aim is to support the choice of two remarkable connections with torsion in a 3-Sasakian manifold, proving that, in contrast to the Levi-Civita connection,  the holonomy group in the homogeneous cases reduces to a proper subgroup of the special orthogonal group, of dimension considerably smaller. We realize the computations of the holonomies in a unified way, by using as a main algebraic tool a nonassociative   structure, that one of  symplectic triple system. 
\end{abstract}

\section{Introduction }

The 3-Sasakian geometry is a natural generalization of the Sasakian geometry introduced independently by Kuo and Udriste in 1970 \cite{Kuo, Udriste}. The 3-Sasakian manifolds are very interesting objects. 
In fact, any 3-Sasakian manifold  has three orthonormal Killing vector fields which span an integrable 3-dimensional distribution. Under some regularity conditions on the corresponding foliation, the space of leaves is a quaternionic K\"ahler manifold with positive scalar curvature. And conversely, over any quaternionic K\"ahler manifold with positive scalar curvature, there is a principal $\SO(3)$-bundle  {that} admits a 3-Sasakian structure \cite{libroGB}. The canonical example is the sphere $\mathbb{S}^{4n+3}$ realized as a hypersurface in $\mathbb H^{n+1}$.
Besides,   any 3-Sasakian manifold is   an Einstein space of positive scalar curvature \cite{kashiwada}.

In spite of that, during the period from 1975 to 1990, approximately, 3-Sasakian manifolds  {lived in} a relative obscurity, probably due to the fact that, according to the authors of the monograph \cite{libroGB},  Boyer and Galicki, the holonomy group of a 3-Sasakian manifold never reduces to a proper subgroup of the special orthogonal group. There was the idea that manifolds should be divided into different classes according to their holonomy group, being \emph{special geometries} those with holonomy group is not of general type. 
Sasakian manifolds in the past.

The revival of the 3-Sasakian manifolds  occurred  in the nineties.   On one hand,   Boyer and Galicki began to study 3-Sasakian geometry because it appeared as a natural object in their quotient construction for certain hyperk\"ahler manifolds \cite{BGM93}. This construction helped to find new collection of examples, and, since then, the topology and geometry of these manifolds  {have been}  continuously studied. On the other hand, during the same period,  the condition of admitting three Killing spinors in a compact Riemannian manifold of dimension 7  {was} proved to be equivalent to the existence of a 3-Sasakian structure \cite{FrieKath}. This  {results in}  a relation between the 3-Sasakian manifolds and the spectrum of the Dirac operator. Recall that Dirac operators have become a powerful tool for the treatment of various problems in geometry, analysis and theoretical physics
\cite{surveyagricola}.   
Much part of this work was developed by the Berlin school around Th.~Friedrich. 

Friedrich himself is who began to study,
in \cite{holonomias}, the holonomy group of connections with skew-symmetric torsion.
Our work on 3-Sasakian manifolds follows this approach. Recall that the holonomy group of the Levi-Civita connection being of general type indicates that  it  is   not well-adapted to the 3-Sasakian geometry, because, when a connection is well-adapted to a particular geometrical structure, then there are parallel tensors and the holonomy reduces. Thus, in this work, we compute the holonomy groups of several affine  connections with skew-symmetric torsion reportedly better adapted to the 3-Sasakian geometry than   the Levi-Civita connection, in order  to support their usage.     
This paper is a natural continuation of the work \cite{nues3Sas}, which looks for  good affine metric connections on a 3-Sasakian manifold, with nonzero torsion.
 The notion of torsion of a connection is due to Elie Cartan \cite{Cartancita}, who investigates several examples of connections preserving geodesics (an equivalent condition to the skew-symmetric torsion), explaining also how the connection should be adapted to the geometry under consideration. Although there are several geometric structures on Riemannian manifolds admitting a unique metric connection preserving the structure with skew-symmetric torsion
 (see references in the complete survey  \cite{surveyagricola} on geometries with torsion),   this is not the situation for a 3-Sasakian manifold: the 3-Sasakian structure is not parallel for any metric connection with skew-torsion. Hence, it is natural to wonder whether there is a \emph{best} affine metric connection on a 3-Sasakian manifold. This is one of the targets in \cite{nues3Sas}, in which some affine connections have been suggested, as the distinguished connection $\nabla^\mathcal S$  in 
\cite[Theorem~5.6]{nues3Sas}, or supported \cite[Theorem~5.7]{nues3Sas}, as the canonical connection $\nabla^c$ defined in \cite{AgriDileo}. 
Both have interesting geometrical properties in any 3-Sasakian manifold,  such as parallelizing the Reeb vector fields or the torsion, respectively.

Now we revisit these affine connections in the homogeneous cases, in which $\nabla^\mathcal S$ is besides the  unique   invariant metric connection with skew-torsion parallelizing the Reeb vector fields (and   $\nabla^c$ is one of the few ones  parallelizing the torsion). 
To be precise,  we compute their curvature operators $R^\mathcal S$ and $R^c$ in terms of an interesting algebraical structure which is hidden behind the 3-Sasakian homogeneous manifolds, namely, that one of symplectic triple system  (Section~\ref{sec_background}). We follow the suggestions in \cite[Remark~4.11]{nues3Sas}. Surprisingly, the expressions of the curvature operators -after complexificacion- turn out to be very simple in terms of these symplectic triple systems. The precise computations of the holonomy algebras (in general a very difficult task)  is not only feasible but also made in a unified way, independently of the considered homogeneous 3-Sasakian manifold. Our main result is essentially the following one:
\begin{quote}
For any simply-connected 3-Sasakian homogeneous manifold  $M^{4n+3}=G/H$, 
\begin{itemize}
\item If $\nabla^{\mathcal S}$ is the distinguished connection, the holonomy group  is isomorphic to $H\times \SU(2)$.
\item If $\nabla^c$ is the canonical connection, the holonomy group  is isomorphic to $H\times \SU(2)$ too (but embedded in $\SO(4n+3)$ in a different way).
\end{itemize}
\end{quote}
We  {provide an additional proof} that if $\nabla^g$ is the Levi-Civita connection, the holonomy group  is the whole group $\SO(4n+3)$. This case is enclosed for completeness, besides providing a nice expression for the curvature operators.

There are no general results about the holonomy group of connections with torsion, in spite of the results on some concrete examples in \cite{holonomias}. For instance, the holonomy algebra on $\mathbb{R}^n$ is   semisimple,  regardless  the considered metric connection with skew-torsion.  A consequence of our results is that the semisimplicity is also a feature of  the  holonomy algebras attached to  {either} $\nabla^{\mathcal S}$  {or} $\nabla^{c}$ if $G\ne\SU(m)$. In contrast, if $G=\SU(m)$, both the holonomy algebras attached to  {either} $\nabla^{\mathcal S}$  {or} $\nabla^{c}$ have a one-dimensional center.

Another precedent on the computation of the holonomy algebras of invariant affine connections with torsion on   homogeneous spaces is \cite{esphomogeneosdeG2}, which deals with the natural connection on the eight reductive homogeneous spaces  of $G_2$, in particular on $G_2/\Sp(1)$. It proves that its related holonomy algebra \cite[Proposition~5.20]{esphomogeneosdeG2} is the whole orthogonal algebra, but take into consideration that  the  metric considered in \cite{esphomogeneosdeG2} is homothetic to the Killing form, so that it is not the  Einstein metric considered in the 3-Sasakian manifold $G_2/\Sp(1)$ (compare with Eq.~\eqref{eq_nuestrag}).  Also, the curvatures on some invariant  connections with skew-torsion on the 3-Sasakian manifold $\mathbb S^7\cong \frac{\Sp(2)}{\Sp(1)}$ have been computed in \cite[\S5.2]{DGP}, including some Ricci-flat not flat affine connections different from $\nabla^{\mathcal S}$ and some flat affine connections on $\mathbb S^7$ (see also \cite{AgriFri2}).\smallskip

Although our objectives in this paper are mainly to use algebraical structures to understand better some aspects of the geometry of the 3-Sasakian manifolds, we can also  read our results in terms of the links between algebra and geometry. Symplectic triple systems began to be studied  in the construction of 5-graded Lie algebras. So, Lie algebras gave birth to this kind of triples. This is not unusual, but it happened also with Freudenthal triple systems and some ternary algebras in the seventies \cite{Yamaguti}. The common origin was to investigate algebraic characterizations of the metasymplectic geometry due to H. Freudenthal, from a point of view of ternary structures (see the works by Meyberg,  Faulkner,  Ferrar and Freudenthal himself). But, afterwards, these structures have become relevant by  their own merits, since they have been used to construct new simple Lie algebras in prime characteristic, new Lie superalgebras, and so on. The philosophy is that, at the end, geometry have influenced the algebraical development of certain algebraical structures.  And vice versa, many algebraical structures -not necessarily binary structures- could help in the study of Differential Geometry. 
\medskip

The paper is organized as follows. The structure of  symplectic triple systems is recalled in Section~\ref{sec_background},  together with   a collection of examples    exhausting the classification of complex simple symplectic triple systems, and their   standard enveloping Lie algebras.
Section~\ref{se_3Sasmanif} recalls the notion of a 3-Sasakian manifold and the classification of the homogeneous ones, relating each 3-Sasakian homogeneous manifold with the corresponding   symplectic triple system (Proposition~\ref{pr_stsasociados}). The distinguished and canonical affine connections are introduced and translated to an algebraical setting, and their geometrical properties are recalled,  {such} as the invariance. Finally, Section~\ref{se_resultados} contains the above mentioned results on the holonomy algebras related to such affine connections. Some comparisons have been added on the related Ricci tensors, scalar curvatures and a table comparing the dimensions of the holonomy groups.

 \section{Background on symplectic triple systems }\label{sec_background}
 
 This subsection is essentially extracted from \cite{alb_triples}.
 \begin{definition}
Let $\FF$ be a field and $T$ an $\FF$-vector space endowed with 
a symplectic form $(\, ,\,  )\colon T\times T\to\mathbb F$, 
  and a triple product $[\, ,\, ,\,]\colon T\times T\times T\to T$.
  It is said that $(T,[\, ,\, ,\,],(\, ,\,  ))$ is a \emph{symplectic triple system} if satisfies 
\begin{align}
[x,y,z]=[y,x,z],\label{eq_uno}\\
[x,y,z]-[x,z,y]=(x,z)y-(x,y)z+2(y,z)x,\label{eq_dos}\\
[x,y,[u,v,w]]=[[x,y,u],v,w]+[u,[x,y,v],w]+[u,v,[x,y,w]],\label{eq_tres}\\
([x,y,u],v)=-(u,[x,y,v]),\label{eq_cuatro}
  \end{align}
  for any $x,y,z,u,v,w\in T$. 
  \end{definition} 
  An \emph{ideal} of the symplectic triple system $T$ is a subspace $I$ of $T$ such that $[T,T,I]+[T,I,T]\subset I$; and the system is said to be  \emph{simple} if $[T,T,T]\ne0$ and it contains no proper ideal. If $\dim T\ne1$ and the field is either $\RR$ or $\CC$, the simplicity of $T$ is equivalent to the nondegeneracy of the bilinear form \cite[Proposition~2.4]{alb_triples}.
  
  There is a close relationship between symplectic triple systems and certain $\ZZ_2$-graded Lie algebras \cite{Yamaguti}. 
  We denote by 
  $$
  d_{x,y}:=[x,y,. ]\in {\rm{End}}_{\FF}(T).
  $$
   Observe that the above set of identities can be read in the following form. 
  By \eqref{eq_cuatro}, $d_{x,y}$ belongs to the symplectic Lie algebra
 $$
  \spf(T,(\, ,\,  ))=\{d\in\gl(T): (d(u),v)+(u,d(v))=0\ \forall u,v\in T\},
  $$
 which is a subalgebra of the general linear algebra $\gl(T)=\mathop{\rm{End}}_{\FF}(T)^-$;
and by \eqref{eq_tres}, $d_{x,y}$ belongs to the   Lie algebra of derivations of the triple, i.e.,
  $$
   \der(T,[\, ,\, ,\, ]):=\{d\in\gl(T): d([u,v,w])=[d(u),v,w]+[u,d(v),w]+[u,v,d(w)]\ \forall u,v,w\in T\},
  $$
  which is also a Lie subalgebra of   $\gl(T)$.
  We call set of \emph{inner derivations} the linear span 
  $$
  \inder(T):=\{\sum_{i=1}^nd_{x_i,y_i}:x_i,y_i\in T\},
  $$
  which is a Lie subalgebra of $\der(T,[\, ,\, ,\, ])$, again taking into account \eqref{eq_tres}. Now, consider
  $(V,\langle.,.\rangle)$ a two-dimensional vector space endowed with a nonzero alternating bilinear form, and the vector space
  $$
  \g(T):=\spf(V,\langle.,.\rangle)\oplus \inder(T)\oplus\, V\otimes T.
  $$
  Then $ \g(T)$ is endowed with a $\ZZ_2$-graded Lie algebra structure (\cite[Theorem~2.9]{alb_triples}) in the following way:
  \begin{itemize}
  \item The usual bracket on the Lie subalgebra $\g(T)_{\bar0}:=\spf(V,\langle.,.\rangle)\oplus \inder(T)$;
  \item The natural action of $\g(T)_{\bar0}$ on $\g(T)_{\bar1}:=  V\otimes T$;
  \item The product of two odd elements is given by
  \begin{equation}\label{eq_sympproducto}
  [a\otimes x,b\otimes y]=(x,y)\gamma_{a,b}+\langle a,b\rangle d_{x,y}\in \g(T)_{\bar0},
\end{equation}
  if $a,b\in V$ and $x,y\in T$, where $\gamma_{a,b}\in\spf(V,\langle.,.\rangle)$ is defined by
  $$
  \gamma_{a,b}:=\langle a,.\rangle b+\langle b,.\rangle a.
  $$
  \end{itemize}
  The algebra $ \g(T)$ is called the \emph{standard enveloping algebra} related to the symplectic triple system $T$. 
  Moreover, $ \g(T)$ is a simple Lie algebra if and only if so is $(T,[\, ,\, ,\,],(\, ,\,))$ (\cite[Theorem~2.9]{alb_triples}).
  
  Let us recall which $\ZZ_2$-graded Lie algebras are the standard enveloping algebras related to some symplectic triple system. As above, take $(V,\langle.,.\rangle)$ a two-dimensional vector space endowed with a nonzero alternating bilinear form. Then, if $\g=\g_{\bar0}\oplus\g_{\bar1}$ is a $\ZZ_2$-graded Lie algebra such that there is $\mathfrak{s}$ an ideal of $\g_{\bar0}$ with $\g_{\bar0}=\spf(V,\langle.,.\rangle)\oplus\mathfrak s$ and there is $T$ an $\mathfrak{s}$-module such that $\g_{\bar1}=V\otimes T$ as $\g_{\bar0}$-module, then most of times $T$ can be endowed with a symplectic triple system structure. To be precise,  the invariance of the Lie bracket in $\g$ under the $\spf(V,\langle.,.\rangle)$-action provides the existence of
  \begin{itemize}
  \item $(\, ,\, )\colon T\times T\to\mathbb F$, alternating,
  \item $d\colon T\times T \to \mathfrak{s}$, symmetric,
  \end{itemize}
  such that Eq.~\eqref{eq_sympproducto} holds for any $x,y\in T$, $a,b\in V$. Now, if $(\, ,\, )\ne0$ and we consider the triple product on $T$ defined by $[x,y,z]:=d_{x,y}(z)\equiv d({x,y})\cdot z\in\mathfrak{s}\cdot T\subset T$, then  $(T,[\, ,\, ,\,],(\, ,\, ))$ is proved to be a  {symplectic triple system}. 
    
  \begin{remark}\label{re_Lts}
  Another algebraical structure involved in our study,  better well-known than the one of symplectic triple system, is that one of \emph{Lie triple system}, which can be identified with the tangent space to a symmetric space. Its relevance now is because $(V\otimes T,[.,.,.])$ is a Lie triple system for the triple product given by
  $$
  [a\otimes x,b\otimes y,c\otimes z]=\gamma_{a,b}(c)\otimes (x,y)z+\langle a,b\rangle c\otimes [x,y,z].
  $$
  Thus   some interesting properties hold, as $\sum_{\tiny
 \begin{array}{l}\textrm{cyclic} \vspace{-4pt}\\a,b,c\vspace{-4pt}\\x,y,z\end{array}} [a\otimes x,b\otimes y,c\otimes z]=0$.
  \end{remark}
  
  Now we describe some important examples of simple symplectic triple systems, extracted from \cite{triples}.

\begin{example}\label{ex_symplectic}
Let $W$ be a vector space over $\FF$ endowed with a nondegenerate alternating bilinear form $(.|.)$. 
Then $T=W$ is a symplectic triple system with the triple product given by
$$
[x,y,z]:=(x,z)y+(y,z)x,
$$
if $x,y,z\in W$. It is easy to check that the algebra of inner derivations $\inder(W)$ is isomorphic to $  \spf(W)$
and the standard enveloping algebra $\g(W)$ is isomorphic to the symplectic Lie algebra $\spf(V\oplus W)$, 
where   the alternating bilinear form is defined by
$$
(v_1+w_1,v_2+w_2)=\langle v_1,v_2\rangle +(w_1|w_2).
$$
This symplectic triple system $W$ is called \emph{of symplectic type}. 
\end{example}

\begin{example}\label{ex_orthogonal}
Let $W$ be a vector space over $\FF$ endowed with a nondegenerate symmetric bilinear form $b\colon W\times W\to\FF$. Then $T=V\otimes W$ is a symplectic triple system with 
$$
\begin{array}{l}
(u\otimes x,v\otimes y):=\frac12\langle u,v\rangle b(x,y),\\
{[}u\otimes x,v\otimes y,w\otimes z]:=\frac12\gamma_{u,v}(w)\otimes b(x,y) z+\langle u,v\rangle w\otimes(b(x,z)y-b(y,z)x),
\end{array}
$$
for any $u,v,w\in V$ and $x,y,z\in W$. Besides,    the algebra of inner derivations $\inder(T)$ is isomorphic to $  \spf(V)\oplus\sof(W)$
and the standard enveloping algebra $\g(T)$ is isomorphic to the orthogonal algebra $\sof((V\otimes V)\oplus W)$,
where   the considered symmetric bilinear form $B$ is
$$
B(v_1\otimes v_2+w_1,v'_1\otimes v'_2+w_2)=\langle v_1,v'_1\rangle \langle v_2,v'_2\rangle+b(w_1,w_2).
$$
The symplectic triple system $T=V\otimes W$ is called \emph{of orthogonal type}. 
\end{example}

\begin{example}\label{ex_special}
Let $W$ be any vector space over $\FF$ and $W^*$ its dual vector space. Then $T=W\oplus W^*$  is a symplectic triple system for the alternating map defined by $(x,y):=0=:(f,g),(f,x):=f(x)$ if $x,y\in W$ and $f,g\in W^*$; and triple product given by
$$
\begin{array}{l}
{[}x,y,T]:=0=:[f,g,T] ,\\
{[}x,f,y]:=f(x)y+2f(y)x,\\
{[}x,f,g]:=-f(x)g-2g(x)f.
\end{array}
$$
 Besides,    the algebra of inner derivations $\inder(T)$ is isomorphic to the general linear algebra $  \gl(W)$
and the standard enveloping algebra $\g(T)$ is isomorphic to the special linear algebra $\slf(V\oplus W)$.
The symplectic triple system $T=W\oplus W^*$ is called \emph{of special type}. 
\end{example}

\begin{example}\label{ex_exceptional}
To simplify, assume $\FF=\CC$ and take either $J=\CC$ or $J=H_3(C^\CC)=\{x=(x_{ij})\in \mathop{\textrm{Mat}}_{3\times3}(C^\CC):x_{ji}=\overline{x_{ij}}\}$ with $C\in\{\RR,\CC,\HH, \OO\}$.\footnote{$C^\CC$ is isomorphic, respectively, to $\CC$, $\CC\times\CC$, $\mathop{\textrm{Mat}}_{2\times2}(\CC)$ and the Zorn algebra.}
 Then the vector space 
$$
T_J=\left\{\begin{pmatrix}\alpha& a\\b&\beta\end{pmatrix}:\alpha,\beta\in\CC,a,b\in J\right\}
$$
becomes a symplectic triple system with the alternating map and triple product given by
$$
\begin{array}{ll}
(x_1,x_2):=\alpha_1\beta_2-\alpha_2\beta_1-t(a_1,b_2)+t(b_1,a_2),\\
\left[x_1,x_2,x_3\right]:=\begin{pmatrix}\gamma(x_1,x_2,x_3)& c(x_1,x_2,x_3)\\-c(x_1^t,x_2^t,x_3^t)&-\gamma(x_1^t,x_2^t,x_3^t)\end{pmatrix},
\end{array}
$$
where, if $x_i=\tiny{\begin{pmatrix}\alpha_i& a_i\\b_i&\beta_i\end{pmatrix}}\in T_J$,
and $x_i^t:=\tiny{\begin{pmatrix}\beta_i& b_i\\a_i&\alpha_i\end{pmatrix}}$, we have
$$
\begin{array}{ll}
 \gamma(x_1,x_2,x_3)= &\big(-3(\alpha_1\beta_2+\beta_1\alpha_2)+t(a_1,b_2)+t(b_1,a_2) \big)\alpha_3\\
 &+2\big( \alpha_1t(b_2,a_3)+\alpha_2t(b_1,a_3) -t(a_1\times a_2,a_3)  \big);\\
c(x_1,x_2,x_3)=&\big(-(\alpha_1\beta_2+\beta_1\alpha_2)+t(a_1,b_2)+t(b_1,a_2) \big)a_3\\
&+2\big((t(b_2,a_3) -\beta_2\alpha_3)a_1+ (t(b_1,a_3) -\beta_1\alpha_3)a_2\big)\\
&+2\big( \alpha_1(b_2\times b_3) +\alpha_2(b_1\times b_3) +\alpha_3(b_1\times b_2)  \big)\\
&-2\big( (a_1\times a_2)\times b_3  +(a_1\times a_3)\times b_2+(a_2\times a_3)\times b_1 \big).
\end{array}
$$
Here, if $J=\CC$, then $t(a,b)=3ab$ and $a\times b=0$. And, if $J=H_3(C^\CC)$,  then $t(a,b)=\frac12\tr(ab+ba)$ and $a\times b=
\frac12\left(ab+ba-\tr(a)b-\tr(b)a+(\tr(a)\tr(b)-t(a,b))I_3  \right)$,
for $I_3\in J$ the identity matrix.\footnote{Thus $(a\times a)\cdot a=n(a)I_3$, for  $n$ the cubic norm and $\cdot$ the symmetrized product $a\cdot b=\frac12(ab+ba)$, so that $\times$ can be seen as a kind of symmetric cross product of the Jordan algebra $(J,\cdot)$.}
Now the pair $(\g(T),\inder(T))$ is described respectively by 
$$
(\mathfrak{g}_2^\CC,\slf(2,\CC)),\quad
(\mathfrak{f}_4^\CC,  \spf(6,\CC) ),\quad
(\mathfrak{e}_6^\CC, \slf(6,\CC) ),\quad
(\mathfrak{e}_7^\CC, \mathfrak{so}({12},\CC) ),\quad
(\mathfrak{e}_8^\CC,\mathfrak{e}_7^\CC).
$$
The symplectic triple system $T_J $ is called \emph{of exceptional type}. 
\end{example}
 According to \cite[Theorem~2.30]{alb_triples}, the described symplectic triple systems in Examples~\ref{ex_symplectic}, \ref{ex_orthogonal}, \ref{ex_special}, and \ref{ex_exceptional}, 
 exhaust the classification of simple symplectic triple systems over the field $\CC$.

\section{3-Sasakian homogeneous manifolds }\label{se_3Sasmanif}

\begin{definition}
A triple $\mathcal{S}=\{\xi, \eta, \varphi\}$ is called a \emph{Sasakian structure} on a Riemannian manifold $(M,g)$ when 
\begin{itemize}
\item $\xi\in \mathfrak{X}(M)$ is a unit Killing vector field (called the \emph{Reeb vector field}), 
\item $\varphi$ is the endomorphism field given by $\varphi (X)=- \nabla^g_{X}\xi$ for all $X\in \mathfrak{X}(M)$, denoting by $\nabla^g$ the  Levi-Civita connection, 
\item $\eta$ is the $1$-form on $M$ metrically equivalent to $\xi$, i.e., $\eta(X)=g(X, \xi)$, \end{itemize}
and the following condition is satisfied
$$
(\nabla^g_{X}\varphi)(Y)=g(X,Y)\xi-\eta (Y) X 
$$
for any $X, Y\in \mathfrak{X}(M)$.
  A \emph{Sasakian manifold} is a Riemannian manifold  $(M,g)$ endowed with a fixed Sasakian structure $\mathcal{S}$. 
  \end{definition}
 
 \begin{definition}
A \emph{$3$-Sasakian structure} on $(M,g)$ is a family of Sasakian structures  $\mathcal{S}=\{\xi_{\tau} ,\eta_{\tau}, \varphi_{\tau}\}_{\tau\in \mathbb{S}^{2}}$ on $(M,g)$ parametri\-zed by   the unit sphere  $ \mathbb{S}^{2}$  and such that, for $\tau, \tau'\in \mathbb{S}^{2}$, the following compatibility conditions hold 
\begin{equation}\label{eq_compatibilityconditions}
g(\xi_{\tau}, \xi_{\tau '})=\tau\cdot \tau'  \quad \mathrm{and}\quad \quad [\xi_\tau, \xi_{\tau'}]=2 \xi_{\tau \times \tau'},
\end{equation}
where  \lq\lq$ \cdot $\rq\rq\ and \lq\lq$\times$\rq\rq\ are the standard inner and cross products in $\mathbb{R}^{3}$,  and   the Reeb vector fields are extended from $ \mathbb{S}^{2}$ to $\mathbb{R}^{3}$ by linearity.
Again if $\mathcal S$ is fixed, $(M,g)$ is said a 3-Sasakian manifold.
 \end{definition}
Recall that having a $3$-Sasakian structure on a Riemannian manifold $(M,g)$ is equivalent to fix three Sasakian structures $\mathcal{S}_{k}=\{\xi_{k} ,\eta_{k}, \varphi_{k}\}$, for $k=1,2,3$, such that $g(\xi_{i},\xi_{j})=\delta_{ij}$ and
$[\xi_{i},\xi_{j}]=2 \epsilon_{ijk}\xi_{k}$.  The compatibility conditions imply,
$$
\begin{array}{c}
\varphi_{k}=\varphi_{i}\circ \varphi_{j }-\eta_{j}\otimes \xi_{i}=-\varphi_{j}\circ \varphi_{i }+\eta_{i}\otimes \xi_{j},\qquad
\xi_{k}=\varphi_i(\xi_{j})=-\varphi_j(\xi_{i}),\\
\end{array}
$$
for any even permutation $(i,j,k)$ of $(1,2,3)$ (i.e., $\epsilon_{ijk}=1$).

The $3$-Sasakian homogeneous manifolds were classified   in \cite{Ale}. Surprisingly,   there is a one-to-one correspondence between   compact simple Lie algebras and   simply-connected $3$-Sasakian homogeneous manifolds.

\begin{theorem}
\label{th_lasSashomogeneas}
Any $3$-Sasakian homogeneous manifold is one of the following coset manifolds:
$$
\frac{\Sp(n+1)}{\Sp(n)},\quad
\frac{\Sp(n+1)}{\Sp(n)\times\mathbb Z_2},\quad
\frac{\SU(m)}{S(\mathrm{U}(m-2)\times \mathrm{U}(1))},\quad
\frac{\SO(k )}{\SO(k-4)\times \Sp(1)},$$
$$\frac{G_2 }{\Sp(1) },\quad  
\frac{F_4 }{\Sp(3) },\quad
\frac{E_6 }{\SU(6) },\quad  
\frac{E_7 }{\mathrm{Spin}(12) },\quad 
\frac{ E_8}{ E_7},
$$
for $n\ge0$, $m\ge3$ and $k\ge7$ ($\Sp(0)$ denoting the trivial group).
\end{theorem}

Strong algebraical implications are suggested by the above result. To our aims, \cite[Remark~4.11]{nues3Sas}  can be read as follows:
\begin{proposition}\label{pr_stsasociados}
If $\g$ and $\hh$ are the Lie algebras of $G$ and $H$, for $M=G/H$ a $3$-Sasakian homogeneous manifold, 
then there is a complex symplectic triple system $T$ such that $(\g^\CC,\hh^\CC)=(\g(T),\inder(T))$. More precisely,  
\begin{enumerate}
\item If $M\in\left\{ \frac{\Sp(n+1)}{\Sp(n)},\ 
\frac{\Sp(n+1)}{\Sp(n)\times\mathbb Z_2}\right\}$ for some $n\ge0$, then $T$ is a symplectic triple system  {of symplectic type} as in Example~\ref{ex_symplectic}. 
\item If $M=\frac{\SU(m)}{S(\mathrm{U}(m-2)\times \mathrm{U}(1))}$ for some $m\ge3$,  then $T$ is a symplectic triple system  {of special type} as in Example~\ref{ex_special}. 

\item If $M=\frac{\SO(k )}{\SO(k-4)\times \Sp(1)}$ for some $k\ge7$,  then $T$ is a symplectic triple system  {of orthogonal type} as in Example~\ref{ex_orthogonal}. 
\item If $M\in\left\{ \frac{G_2 }{\Sp(1)},\  
\frac{F_4 }{\Sp(3)},\ 
\frac{E_6 }{\SU(6)},\   
\frac{E_7 }{\mathrm{Spin}(12)},\  
\frac{ E_8}{ E_7}
\right\},$ then $T$ is a symplectic triple system  {of exceptional type} as in Example~\ref{ex_exceptional}. 
\end{enumerate}
\end{proposition}

We want to apply the algebraical structure of symplectic triple systems to make easy the computations on curvatures and holonomies. First, we recall Nomizu's Theorem  on invariant affine connections on reductive homogeneous spaces \cite{teoNomizu}, since, in this setting, the curvature and torsion tensors are easily written in algebraical terms.

Let $G$ be a Lie group acting transitively on a manifold $M$. 
An affine connection $\nabla$ on $M$ is said to be $G$-\textit{invariant}  if, for each $\sigma \in G$ and for all $X,Y\in\mathfrak{X}(M)$,
$$
\tau_{\sigma}(\nabla_{X}Y)=\nabla_{_{\tau_{\sigma}(X)}}\tau_{\sigma}(Y).
$$
Here 
$\tau_\sigma\colon M\to M$ denotes the diffeomorphism given by the action, $\tau_\sigma(p)=\sigma\cdot p$ if 
$p\in M$; and
the vector field $\tau_{\sigma}(X)\in \mathfrak{X}(M)$ is defined  by
$
 (\tau_{\sigma}(X) )_p:=(\tau_{\sigma})_{*}(X_{\sigma^{-1}\cdot p})
$
at each $p\in M$.
If $H$ is the isotropy subgroup at a fixed point $o\in M$, then there exists a diffeomorphism between $M$ and $G/H$. 
For $H$ connected,
the  homogeneous space $M=G/H$    is said to be \textit{reductive} if the Lie algebra $\mathfrak{g}$ of $G$ admits a vector space decomposition
$\mathfrak{g}=\mathfrak{h}\oplus\mathfrak{m}$,
for $\mathfrak{h}$ the Lie algebra of $H$ and $\mathfrak{m}$ an $\hh$-module. 
The differential map $\pi_{*}$ of the projection $\pi\colon G\to M=G/H$ gives a linear isomorphism
$(\pi_{*})_e\vert_{\mathfrak{m}}\colon\mathfrak{m}\to T_{o}M$, where $o=\pi(e)$.
Under these conditions, Nomizu's Theorem asserts:

\begin{theorem}\label{nomizu}
There is a one-to-one correspondence between the set of   $G$-invariant affine connections $\nabla$ on the reductive homogeneous space $M=G/H$ and the vector space of bilinear maps $\alpha\colon \mathfrak{m}\times\mathfrak{m}\to\mathfrak{m}$ such that $\hh\subset\der(\mathfrak{m},\alpha)$.

The torsion and curvature tensors of $\nabla$ can then be computed in terms of the related map $\alpha^{_\nabla}$ as follows:
$$
\begin{array}{l}
T^{\nabla}(X,Y)= \ \alpha^{_\nabla} (X,Y)-\alpha^{_\nabla} (Y,X)-[X,Y]_\mathfrak{m}, \\
 R^{\nabla}(X,Y)Z=  
\alpha^{_\nabla} (X,\alpha^{_\nabla}(Y, Z))-\alpha^{_\nabla}(Y,\alpha^{_\nabla} (X, Z)) 
  -\alpha^{_\nabla}( [X , Y]_{\mathfrak{m}}, Z)-[[X , Y]_{\mathfrak{h}} , Z],
  \end{array}
$$
for any $X, Y, Z\in \mathfrak{m}$, where   $[\ ,\ ]_{\mathfrak{h}}$ and   $[\ ,\ ]_{\mathfrak{m}}$ denote the composition of the bracket
($ [\mathfrak{m},\mathfrak{m}]\subset \mathfrak{g}$) with  the projections 
$\pi_{\mathfrak{h}},\pi_{\mathfrak{m}}\colon \g\to\g$  of $\mathfrak{g}= \mathfrak {h}\oplus \mathfrak{m}$ on each summand, respectively. 
 \end{theorem}

Note that the above expressions give $T^{\nabla}$  and $R^{\nabla}$ at the point $o=\pi(e)$, but the invariance permits to recover the whole tensors on $M$.  
In particular, every  homogeneous 3-Sasakian manifold (all of them described in Theorem~\ref{th_lasSashomogeneas}) is a reductive homogeneous space, 
and its invariant affine connections have been thoroughly  studied in \cite{nues3Sas} and described in terms of the related
  reductive decomposition $\mathfrak{g}= \mathfrak {h}\oplus \mathfrak{m}$. Note also that 
$$
\g=\g_{\bar0}\oplus \g_{\bar1}$$
is in any case a $\ZZ_2$-graded Lie algebra such that
$$
\g_{\bar0}=\hh\oplus\spf(1) \quad\textrm{ and }\quad
\mm=\spf(1)\oplus\g_{\bar1}.
$$
The (invariant) metric $g$ is determined by $ g_o\colon T_oM\times T_oM\to\RR$, which, under the identification $(\pi_{*})_e\vert_{\mathfrak{m}}$, is given by 
\begin{equation}\label{eq_nuestrag}
g\vert_{\spf(1)}=- \frac1{4(n+2)} \kappa,\quad g\vert_{\g_{1} }=-\frac1{8(n+2)}\kappa,\quad g\vert_{\spf(1)\times\g_{1} }=0,
\end{equation}
for $\kappa$ the Killing form of $\g$ (\cite[Theorem~4.3ii)]{nues3Sas}).\smallskip

Let us stand out  certain invariant affine connections, joint with their related bilinear maps through Nomizu's theorem, which have been distinguished in  \cite{nues3Sas} because of satisfying certain geometrical properties.

\begin{example}\label{ex_LC}
The Levi-Civita connection $\nabla^g$ is related  to the   bilinear map $\alpha^g\colon\mm\times\mm\to\mm$, defined as follows:
 \begin{equation}\label{eq_alfadeLevi}
 \alpha^g(X,Y)=\left\{\begin{array}{ll}
 0&\text{if $\,X\in\spf(1)$ and $\,Y\in\g_{\bar1}$},\\
 \frac12[X,Y]_\mm&\text{if either $\,X,Y\in\spf(1)$ or $\,X,Y\in\g_{\bar1}$},\\
{ [X,Y]}_\mm&\text{if $\,X\in\g_{\bar1}$ and $\,Y\in\spf(1)$},\\
 \end{array}
 \right.
\end{equation}
according to \cite[Theorem~4.3]{nues3Sas}.
Note that $\alpha^g$ is not a skew-symmetric map.
\end{example}

If an affine connection is compatible with the metric, that is, the tensor $g$ is parallel ($\nabla g=0$), then the connection is determined by its torsion, with $\nabla=\nabla^g+\frac12T^\nabla$. Algebraically, an invariant affine connection is metric if and only if $\alpha^{_\nabla}(X,.)\in\sof(\mm,g)$ for all $X\in\mm$. Among the metric affine connections, 
we have studied those with totally skew-symmetric torsion, or briefly, \emph{skew-torsion}, that is, those ones where
\begin{equation}\label{tresforma}
\omega_{_\nabla}(X,Y,Z):=g(T^{\nabla}(X,Y),Z)
\end{equation}
is a differential $3$-form on $M$.   This  characterizes the remarkable fact that $\nabla$ and $\nabla^g$ share their (parametrized) geodesics. In   terms of the related bilinear map $\alpha^{_\nabla}$, the characterization is 
that the trilinear map 
$g(( \alpha^{_\nabla}-\alpha^g)(.,.),.)\colon\mm\times\mm\times\mm\to\RR$ is  alternating, so providing an element in $\hom_\hh(\Lambda^3(\mm),\RR)$.

Denote by  $\{\xi_i\}_{i=1}^3$  the $G$-invariant   vector fields on $M$ corresponding to the following basis of $\spf(1)=\suf(2)$, 
\begin{equation}\label{eq_losxis}
\xi_1=\left(\begin{array}{cc}\mathbf{i}&0\\0&-\mathbf{i} \end{array}\right),\quad
\xi_2=\left(\begin{array}{cc}0&-1\\ 1&0 \end{array}\right),\quad
\xi_3=\left(\begin{array}{cc}0&-\mathbf{i}\\ -\mathbf{i}&0 \end{array}\right).
\end{equation}
Then, according to \cite[Theorem~4.3]{nues3Sas}, the endomorphism field $\varphi_{i}=-\nabla^g \xi_{i}$ satisfies
\begin{equation*}\label{derivada de xi}
\varphi_i\vert_{\spf(1)}=\frac12\ad \xi_i,\quad \varphi_i\vert_{\g_{\bar1} }= \ad \xi_i ,
\end{equation*}
for each $i=1,2,3$, and   $\mathcal{S}_i=\{\xi_i,\eta_i,\varphi_i\}$   is a Sasakian structure for $\eta _i=g(\xi_{i},-)$ and $g\colon\mm\times\mm\to\RR$ given by Eq.~\eqref{eq_nuestrag}.  

Denote by $\alpha_o,\alpha_{rs}\colon \mm\times\mm\to\mm$  the alternating bilinear maps determined by the nondegeneracy of the metric $g$ and  
\begin{equation}\label{def_alfas}
 \eta_1\wedge\eta_2\wedge\eta_3(X,Y,Z)=g(\alpha_o(X,Y),Z),\quad 
\eta_r\wedge\Phi_s(X,Y,Z)=g(\alpha_{rs}(X,Y),Z),
\end{equation}
for any $X,Y,Z\in\mathfrak m$ and $r,s \in \{1,2,3\}$, where $\Phi_s(X,Y)=g(X,\varphi_s(Y)) $ is a 2-form. A key result \cite[Corollary~5.3]{nues3Sas} states that the set of bilinear maps related  by Theorem~\ref{nomizu} to the invariant (metric) affine connections with skew-torsion  coincides with the set
\begin{equation}\label{eq_skewtor}
\{\alpha^g+a\alpha_o+\sum_{r,s=1}^3b_{rs}\alpha_{rs}:a,b_{rs}\in\RR\},
\end{equation}
for any 3-Sasakian homogeneous manifold of dimension at least 7 except for the case $M= \frac{\SU(m)}{S(\mathrm{U}(m-2)\times \mathrm{U}(1))} $,   in which the set of bilinear maps related to the invariant   affine connections with skew-torsion contains strictly the set in \eqref{eq_skewtor}.
For further use, recall from \cite[Eq.~(40)]{nues3Sas},
\begin{equation}\label{eq_torsiones}
\begin{array}{ll}
\alpha_o(X,Y)=0,\qquad&  \alpha_{rs}(X,Y) =\Phi_s(X,Y)\xi_r,        \\
  \alpha_o(X,\xi_j)=0,&   \alpha_{rs}(X,\xi_j)= \delta_{rj}\varphi_sX,       \\
  \alpha_o(\xi_i,\xi_{i+1})=\xi_{i+2},\qquad\quad&   \alpha_{rs}(\xi_i,\xi_{i+1})=-\delta_{rs}\xi_{i+2},   \\
\end{array}
\end{equation}
for any $X,Y\in\mathfrak{g}_{\bar1}$.

\begin{example}\label{ex_canonical}
According to \cite[Theorem~4.1.1]{AgriDileo}, the \emph{canonical connection} of a 3-Sasakian manifold is the (metric)  affine connection with skew-torsion  characterized by 
\begin{equation}\label{eq_laquesalemucho}
\nabla^c_X\varphi_i=-2(\eta_{i+2}(X)\varphi_{i+1}-\eta_{i+1}(X)\varphi_{i+2}).
\end{equation}
 Its torsion $T^c$ is determined by the 3-form
\begin{equation*} 
\omega^c= \sum_{i=1}^3 \eta_{i}\wedge d\eta_{i},
\end{equation*}
as in \cite[Remark~5.18]{nues3Sas}. A key property is that it has parallel torsion ($\nabla^cT^c=0$). Take into account  that, for dimension of the homogeneous 3-Sasakian manifold strictly bigger than 7,   the Levi-Civita connection, the characteristic connection of any of the involved Sasakian structures and $\nabla^c$ are the only invariant affine connections with parallel skew-torsion (\cite[Theorem~5.7]{nues3Sas}).

The related  bilinear map $\alpha^c\colon\mm\times\mm\to\mm$ is
$$
\alpha^c=\alpha^g+\sum_{r=1}^3\alpha_{rr},
$$
for $\alpha_{rr}$  the skew-symmetric bilinear map defined in Eq.~\eqref{def_alfas}.
\end{example}

\begin{example}\label{ex_dist}
Let   $\mathcal S=\{\xi_{\tau} ,\eta_{\tau}, \varphi_{\tau}\}_{\tau\in \mathbb{S}^{2}}$ be  a 3-Sasakian structure on 
   a  3-Sasakian homogeneous manifold $(M,g)$ of dimension at least $7$.
The unique $G$-invariant affine connection    with skew-torsion on $M$  satisfying that   $\xi_\tau$ is parallel for any $\tau\in \mathbb{S}^{2}$, is denoted by
 $\nabla^\mathcal{S}$. The related bilinear map is given by
 $$
\alpha^\mathcal{S}= \alpha^g+2\alpha_o+ \sum_{r=1}^3 \alpha_{rr}. 
$$ 
The above connection $\nabla^\mathcal{S}$ admits trivially a generalization for  not necessarily homogeneous 3-Sasakian manifolds, and   then $\nabla^\mathcal{S}$ becomes
\begin{itemize}
\item  Einstein   with skew-torsion, with symmetric Ricci tensor, if $\dim M=7$;
\item $\mathcal S$-Einstein, for arbitrary dimension.
\end{itemize}
The concept of {Einstein   with skew-torsion} is introduced in \cite{AgriFerr}, where it is proved that the metric connections   
such that $(M,g,\nabla)$ is Einstein with skew-torsion are the critical points  of certain variational problem. 
 An affine connection with totally skew-symmetric torsion is \emph{Einstein with skew-torsion} if the symmetric part of the Ricci tensor is multiple of the metric, while is \emph{$\mathcal S$-Einstein} \cite[Definition~5.2]{nues3Sas} when the Ricci tensor is  multiple of the metric both in the horizontal and vertical distributions. There are  $\mathcal S$-Einstein invariant affine connections in any homogeneous 3-Sasakian manifold but there are invariant affine connections Einstein with skew-torsion only if the homogeneous 3-Sasakian manifold has dimension 7 \cite[Theorem~5.4i)]{nues3Sas}.
 \end{example}

Our purpose is to study the holonomy groups of $\nabla^g$, $\nabla^\mathcal S$, and $\nabla^c$, in order to figure out which of these connections is in some way better adapted to the 3-Sasakian geometry. A good sign of an affine  connection   adapted to a geometry is that the holonomy group is small.

  \section{Curvatures and holonomies}\label{se_resultados}
  
 We begin by recalling some well-known facts in order to unify the notation. Given a piecewise smooth loop $\gamma\colon[0,1]\to M$  based at a point $p\in M$, a connection $\nabla$ defines a parallel transport map $P_\gamma\colon T_pM\to T_pM$, which is both linear and invertible. The holonomy group of $\nabla$ based at $p$ is defined as
$\operatorname {Hol} _{p}(\nabla )=\{P_{\gamma }\in \mathrm {GL} (T_pM)\mid \gamma {\text{ is a loop based at }}p\}.$ In our case the holonomy group does not depend  on the basepoint (up to conjugation) since $M$ is connected, and $\operatorname {Hol}  (\nabla )=\operatorname {Hol} _{o}(\nabla )$ turns out to be a Lie group  which can be identified with a subgroup of $\GL(\mm)$.
Ambrose-Singer Theorem \cite{A-S} gives a way of computing the holonomy group in terms of the curvature tensor of the connection. 
The Lie algebra of the holonomy group $\operatorname {Hol}  (\nabla )$, denoted by $\hol(\nabla)$ and called the \emph{holonomy algebra}, turns out to be the smallest Lie subalgebra of $\gl(\mm)$ containing the curvature tensors $R^{\nabla}(X,Y)$ for any $X, Y\in \mathfrak{m}$ and closed under commutators with the    left multiplication operators $\alpha^{_\nabla}(X,.)$, for $X\in \mm$.    If   $\nabla$ is compatible with the (invariant) metric $g$, we already mentioned that $\alpha^{_\nabla}(X,.)\in\sof(\mm,g)$ for all $X\in\mm$, so that $R^{\nabla}(X,Y)\in\sof(\mm,g)$ too, and hence the holonomy algebra is a subalgebra of the orthogonal Lie algebra.  

Our main tool in this section will be the complexification, since the complex Lie algebra $\hol(\nabla)^\CC$ is the smallest Lie subalgebra of $\gl(\mm^\CC)$ containing the curvature maps for any $X, Y\in \mathfrak{m}^\CC$,
\begin{equation}\label{eq_R}
R^{\nabla}(X,Y)=[\alpha^{_\nabla}_X,\alpha^{_\nabla}_Y]-\alpha^{_\nabla}_{[X,Y]_{\mm}}-\ad[X,Y]_{\hh},
\end{equation} 
  and closed under commutators with the operators of left multiplications $\alpha^{_\nabla}_X=\alpha^{_\nabla}(X,.)$ with $X\in \mm^\CC$  (we use the same notation for $R^{\nabla}$, $\alpha^{_\nabla}$, $[.,.]_{\mm}$ and $[.,.]_{\hh}$,  to avoid complicating the notation).    Our setting is $\mm^\CC=\spf(V,\langle.,.\rangle)\oplus\, V\otimes T$ and $\hh^\CC= \inder(T)$, for $T$ a complex symplectic triple system and  $(V,\langle.,.\rangle)$ a two-dimensional complex vector space endowed with a nonzero alternating bilinear form. This makes very easy to compute the curvature maps. 
From Eq.~\eqref{eq_sympproducto}, the projections of the bracket on $\hh^\CC$ and $\mm^\CC$  are
 $$
 \begin{array}{l}
  [\xi+a\otimes x,\xi'+b\otimes y]_\hh=\langle a,b\rangle d_{x,y} ,
 \\
  {[}\xi+a\otimes x,\xi'+b\otimes y]_\mm=
[\xi,\xi']+(x,y)\gamma_{a,b}+\xi(b)\otimes y-\xi'(a)\otimes x,
  \end{array}
 $$ 
for any $\xi,\xi'\in  \spf(V,\langle.,.\rangle)$, $a,b\in V$, $x,y\in T$.
(Recall the definitions of $\gamma_{a,b}\in  \spf(V,\langle.,.\rangle) $ and $d_{x,y}\in\inder(T) $ in Section~\ref{sec_background}.)
  
  \subsection{ The Levi-Civita connection}\label{subsec_LC}
  We enclose this case for completeness, and for providing an unified treatment, including the algebraical expressions of the curvature operators. The fact that the   holonomy group is general is well-known (recall the words -mentioned in the Introduction- of Boyer and Galicki in \cite{libroGB}, or see, for instance, \cite[Corollary~14.1.9]{libroGB}) but we have not been able to find a direct and explicit proof in the literature.

 \begin{proposition}\label{prR^g}
After complexifying,  the curvature operators become
$$
\begin{array}{l}
R^g(\xi,\xi') (\xi''+a\otimes x)=-\frac14[[\xi,\xi'],\xi''],\vspace{3pt}\\
R^g(a\otimes x, \xi)(\xi'+b\otimes y)=
-\frac12(x,y)\langle a,b\rangle \xi+g(\xi,\xi')a\otimes x,\vspace{4pt}\\
R^g(a\otimes x, b\otimes y)(\xi+c\otimes z)=\frac{\gamma_{a,c}(b)\otimes (x,z)y-\gamma_{b,c}(a)\otimes (y,z)x}2-\langle a,b \rangle c\otimes [x,y,z],
\end{array}
$$
for any $\xi,\xi',\xi''\in  \spf(V,\langle.,.\rangle)$, $a,b,c\in V$, $x,y,z\in T$.
\end{proposition}

\begin{proof}
These computations are straightforward, by taking in mind Eq.~\eqref{eq_alfadeLevi}. First, we compute
$$
\begin{array}{l}
R^g(a\otimes x, b\otimes y)(c\otimes z)\\
=\alpha^g_{a\otimes  x}\big(\frac12(y,z)\gamma_{b,c}\big)
-\alpha^g_{b\otimes  y}\big(\frac12(x,z)\gamma_{a,c}\big)
-\alpha^g_{(x,y)\gamma_{a,b}}\big(c\otimes z)-[\langle a,b \rangle d_{x,y},c\otimes z]\\
=-\frac12(y,z)\gamma_{b,c}(a)\otimes x+\frac12(x,z)\gamma_{a,c}(b)\otimes y-0-\langle a,b \rangle c\otimes [x,y,z];
\end{array}
$$
and also
$$
\begin{array}{l}
R^g(a\otimes x, b\otimes y)(\xi)= 
\alpha^g_{a\otimes  x}\big( -\xi(b)\otimes y\big)
-\alpha^g_{b\otimes  y}\big( -\xi(a)\otimes x\big)
-\alpha^g_{(x,y)\gamma_{a,b}}\big( \xi\big)-0\\
=-\frac12(x,y)\gamma_{a,\xi(b)}+\frac12(y,x)\gamma_{b,\xi(a)}-(x,y)\frac12[\gamma_{a,b},\xi]\\
=\frac12(x,y)\big( -\gamma_{a,\xi(b)} -\gamma_{b,\xi(a)}+[\xi,\gamma_{a,b}] \big)=0.
\end{array}
$$
Another curvature operator is
$$
\begin{array}{l}
R^g(\xi,\xi') (\xi''+a\otimes x)=\alpha^g_\xi\big(\frac12[\xi',\xi'']\big)-\alpha^g_{\xi'}\big(\frac12[\xi,\xi'']\big)-\alpha^g ( [\xi,\xi'] ,\xi''+a\otimes x)\\
=\frac14\big( [\xi,[\xi',\xi'']]-[\xi',[\xi,\xi'']]  \big)-\frac12[[\xi,\xi'],\xi'']=\big(\frac14-\frac12\big)[[\xi,\xi'],\xi''].
\end{array}
$$
Finally, recall $\alpha_\xi^g(\g_{\bar1})=0$ to get
$$
\begin{array}{l}
R^g(a\otimes x, \xi)(\xi'+b\otimes y)
=\alpha^g_{a\otimes x}(\frac12[\xi,\xi'])-\alpha^g_{\xi}(\frac12(x,y)\gamma_{a,b})
-\alpha^g_{-\xi(a)\otimes x}(\xi'+b\otimes y)\\
=-\frac12[\xi,\xi'](a)\otimes x-\frac14(x,y)[\xi,\gamma_{a,b}]-\xi'\xi(a)\otimes x+\frac12(x,y)\gamma_{\xi(a),b}
\vspace{3pt}\\
=-\frac12(\xi\xi'+\xi'\xi)(a)\otimes x+\frac{(x,y)}2(\frac{-\gamma_{\xi(a),b}-\gamma_{a,\xi(b)}}2+\gamma_{\xi(a),b})\vspace{3pt}\\
=g(\xi,\xi')a\otimes x+\frac14(x,y)(\gamma_{\xi(a),b}-\gamma_{a,\xi(b)}).
\end{array}
$$
In the last step we have used    $\xi\xi'+\xi'\xi=-2g(\xi,\xi')\id_V$ by Eq.~\eqref{eq_losxis}, because two different elements in the orthonormal basis $\{\xi_i\}_{i=1}^3$ anticommute. 
For simplifying the  expression, note that,
 for any $a,b,c\in V$,
 \begin{equation}\label{eq_key}
\langle a,b\rangle c+\langle b,c\rangle a+\langle c,a\rangle b=0,
 \end{equation}
 because $\dim V=2$ (we can assume $c\in\{a,b\}$ by trilinearity) and $\langle .,.\rangle $ is alternating. 
 Thus, using $\langle \xi(a),c\rangle +\langle a,\xi(c)\rangle =0$ when $\xi$ belongs to the symplectic Lie algebra, and Eq.~\eqref{eq_key},
 $$
 \begin{array}{l}
 \big(\gamma_{\xi(a),b}-\gamma_{a,\xi(b)}\big)(c)=\langle \xi(a),c\rangle b+\langle b,c\rangle \xi(a)-\langle a,c\rangle \xi(b)-\langle \xi(b),c\rangle a\\
 =-\langle a,\xi(c)\rangle b+\xi\big( \langle b,c\rangle a-\langle a,c\rangle b\big)+\langle b,\xi(c)\rangle a\\
 =\langle \xi(c),a\rangle b+\xi(-\langle a,b\rangle c)+\langle b,\xi(c)\rangle a= -\langle a,b\rangle \xi(c)-\langle a,b\rangle \xi(c).
 \end{array}
$$
So $\gamma_{\xi(a),b}-\gamma_{a,\xi(b)}=-2\langle a,b\rangle \xi$.
\end{proof}
  
  Let us introduce some convenient notation. Take $
  \varphi_{a,b}:=g( a,.) b-g( b,.) a.
  $ As $g\colon \mm\times \mm\to\RR$ is a bilinear symmetric map, then $ \varphi_{a,b}$ belongs to $\sof(\mm,g)$. Moreover, these maps span the whole orthogonal Lie algebra, i.e., $\sof(\mm,g)=\{\sum_i \varphi_{a_i,b_i}:
  a_i,b_i\in\mm\}\equiv\varphi_{\mm,\mm}$.
   If $\mm=\mm_{\bar0}\oplus\mm_{\bar1}$ is a  decomposition as a sum of vector subspaces,   the general Lie algebra $\gl(\mm)$ admits a $\ZZ_2$-grading
  $\gl(\mm)=\gl(\mm)_{\bar0}\oplus\gl(\mm)_{\bar1}$ where 
  $\gl(\mm)_{\bar i}=\{f\in \gl(\mm):f(\mm_{\bar j})\subset \mm_{\bar i+\bar j}\,\forall\bar j=\bar0,\bar1\}$ for any ${\bar i}\in\ZZ_2$.
  If such vector space decomposition is orthogonal for $g$, then the orthogonal Lie algebra
  $\sof(\mm,g)$ inherits the $\ZZ_2$-grading,
  $\sof(\mm,g)=\sof(\mm,g)_{\bar0}\oplus\sof(\mm,g)_{\bar1}$, being 
  $\sof(\mm,g)_{\bar i}=\gl(\mm)_{\bar i}\cap \sof(\mm,g)$.
  Then it is clear that 
  \begin{equation}\label{eq_grading}
  \sof(\mm,g)_{\bar 0}=\varphi_{\mm_{\bar0},\mm_{\bar0}}+\varphi_{\mm_{\bar1},\mm_{\bar1}},\qquad
  \sof(\mm,g)_{\bar 1}=\varphi_{\mm_{\bar0},\mm_{\bar1}}.
\end{equation}

  We are going to apply the above to
  $$\mm_{\bar0}=\spf(1)=\g_{\bar0}\cap \mm \quad\textrm{ and }\quad\mm_{\bar1}= \g_{\bar1},$$
  since they are orthogonal relative to the Killing form, and then, orthogonal for $g$.
  An interesting remark is that $\alpha^g_{\mm_{\bar 0}}\subset \gl(\mm)_{\bar 0}$ and 
  $\alpha^g_{\mm_{\bar 1}}\subset \gl(\mm)_{\bar 1}$. Note also that 
  $[\mm_{\bar i},\mm_{\bar j}]\subset \g_{\bar i+\bar j}$, 
  and that $\pi_\hh,\pi_\mm\colon\g\to\g$ are grade preserving maps, so that
  $R^g(\mm_{\bar i},\mm_{\bar j},\mm_{\bar k})\subset\mm_{\bar i+\bar j+\bar k}$, or equivalently,
  $R^g(\mm_{\bar i},\mm_{\bar j})\subset\gl(\mm)_{\bar i+\bar j}$.
  Thus, as  the curvature operators are orthogonal maps, then
  $R^g(\mm_{\bar i},\mm_{\bar j})\subset  \sof(\mm,g)_{\bar i+\bar j}$.\smallskip

With the introduced notation, the condition, for a Riemannian manifold $(M,g)$, of being a Sasakian manifold is equivalent to the existence of a Killing vector field $\xi$ of unit length so that $\nabla^g_X\varphi=-\varphi_{X,\xi}$, and  is also equivalent to the existence of a Killing vector field $\xi$ of unit length such that $R^g(X,\xi)=-\varphi_{X,\xi}$, according to \cite[Proposition~1.1.2]{BG}. This is the main clue in the next theorem. Taking the advantage we have computed the concrete expressions of the curvature operators in the above proposition, we can check directly   that, for any $\xi,\xi'\in  \spf(V,\langle.,.\rangle)$, $a\in V$, $x\in T$, 
\begin{align}\label{eq_operadoresbuenos2}
R^g(\xi, \xi')=-\varphi_{\xi,\xi'},\\\label{eq_operadoresbuenos}
  R^g(a\otimes x, \xi)=-\varphi_{ a\otimes x,\xi}.
  \end{align}

Indeed,   $R^g(\xi, \xi')\vert_{\g_{\bar1}}=0=-\varphi_{\xi,\xi'}\vert_{\g_{\bar1}}$, and both $R^g(\xi_i, \xi_{i+1})\vert_{\spf(1)}=-\frac12\ad\xi_{i+2}$ and  $-\varphi_{\xi_i, \xi_{i+1}}=g( \xi_{i+1},.) \xi_{i}-g( \xi_{i},-) \xi_{i+1}$
send
$$\xi_{i}\to\xi_{i+1},\qquad \xi_{i+1}\to-\xi_{i},\qquad \xi_{i+2}\to 0.$$ 
In order to prove \eqref{eq_operadoresbuenos}, we would need to know if the
complexification of the Killing form restricted to the odd part of the grading is
$\kappa(a\otimes x,b\otimes y)=-4(n+2)\,\langle a,b\rangle\, (x,y)$,
because
$-\varphi_{ a\otimes x,\xi}(\xi'+b\otimes y)=
g(\xi,\xi')a\otimes x+\frac1{8(n+2)}\kappa(a\otimes x,b\otimes y)\xi.$
Without doing explicit computations on traces,
classical arguments of representation theory provides the existence of $s\in\CC^\times$ such that 
  \begin{equation}\label{le_kappa}
  \kappa(a\otimes x,b\otimes y)=s\,\langle a,b\rangle\, (x,y)
  \end{equation}
  for all $a,b\in V$, $x,y\in T$. By the associativity of $\kappa$, 
the restriction $\kappa\colon{\g_{\bar1}^\CC}\times {\g_{\bar1}^\CC}\to\CC$ is a ${\g_{\bar0}^\CC}$-invariant bilinear symmetric form, which permits to identify ${\g_{\bar1}^\CC}$ with $({\g_{\bar1}^\CC})^*$. Also the map   given   by $(a\otimes x,b\otimes y)\to(x,y)\langle a,b\rangle $ is a nonzero ${\g_{\bar0}^\CC}$-invariant bilinear symmetric form, which provides another identification between ${\g_{\bar1}^\CC}$ with $({\g_{\bar1}^\CC})^*$, and hence, when composing, an element in $\hom_{\g_{\bar0}^\CC}(\g_{\bar1}^\CC,\g_{\bar1}^\CC)$, which coincides with  $\CC\id$ by Schur's Lemma (\cite[Lemma~3.4]{nues3Sas}), since $\g_{\bar1}^\CC$ is an irreducible $\g_{\bar0}^\CC$-module. This provides the required $s$.
Consider now the map $\rho= R^g(a\otimes x, \xi)+\varphi_{ a\otimes x,\xi}\in \sof(\mm^\CC,g)$, which satisfies $\rho\vert_{\spf(V,\langle.,.\rangle)}=0$ and
 $\rho\vert_{V\otimes T}=\left(\frac{4(n+2)}{s}+1\right)\varphi_{ a\otimes x,\xi}$. Hence we have
 $$
 0=g(\rho(\xi),b\otimes y)+g(\xi,\rho(b\otimes y))=\left(\frac{4(n+2)}{s}+1\right)g(a\otimes x,b\otimes y)g(\xi,\xi),
 $$
 which implies $s=-4(n+2)$, $\rho=0$, and Eq.~\eqref{eq_operadoresbuenos}.
 Now, it is not difficult to find the holonomy algebra.

 \begin{theorem}\label{teo_holg}
The complexification of the holonomy Lie algebra  of the Levi-Civita connection is
$
\hol(\nabla^g)^\CC=\sof(\mm^\CC,g);
$
so that
$$\hol(\nabla^g)=\sof(\mm,g).$$
\end{theorem}

\begin{proof}
Recall that $\hol(\nabla^g)^\CC$ is a Lie subalgebra of $\sof(\mm^\CC,g)$.
  Equation~\eqref{eq_operadoresbuenos} implies that    $R^g(\mm^\CC_{\bar 1},\mm^\CC_{\bar0})= \sof(\mm^\CC,g)_{\bar 1}$   
by taking into account Eq.~\eqref{eq_grading},
 so that   $\sof(\mm^\CC,g)_{\bar 1}$ is contained in $\hol(\nabla^g)^\CC$. In particular, also $[\sof(\mm^\CC,g)_{\bar 1},\sof(\mm^\CC,g)_{\bar 1}]\subset\hol(\nabla^g)^\CC$. We compute
\begin{equation}\label{eq_cuentasholg}
 \begin{array}{ll}
 [\varphi_{\xi,a\otimes x},\varphi_{\xi',b\otimes y}]&=
 \varphi_{\varphi_{\xi,a\otimes x}(\xi'),b\otimes y}+\varphi_{\xi',\varphi_{\xi,a\otimes x}(b\otimes y)}\\
 &=g(\xi,\xi')\varphi_{a\otimes x,b\otimes y}-g(a\otimes x,b\otimes y)\varphi_{\xi',\xi}\in\hol(\nabla^g)^\CC.
 \end{array}
\end{equation}

But  Eq.~\eqref{eq_operadoresbuenos2}
gives $\varphi_{\xi',\xi}\subset  \hol(\nabla^g)^\CC$.
By making the sum with the map in Eq.~\eqref{eq_cuentasholg}, we also have $\varphi_{a\otimes x,b\otimes y}\in\hol(\nabla^g)^\CC$. Hence, 
$\sof(\mm^\CC,g)_{\bar 0}=\varphi_{\mm_{\bar0},\mm_{\bar0}}+\varphi_{\mm_{\bar1},\mm_{\bar1}}$ is completely contained in $ \hol(\nabla^g)^\CC$, what ends the proof.
\end{proof}

\begin{remark}
Note that   the 3-Sasakian homogeneous manifolds are $n$-Sasakian manifolds for $n=3$ in the sense proposed by \cite{nSas}, that is, $R^g(X,Y)\vert_{\spf(1)}=-\varphi_{X,Y}$ for any $X,Y\in\mm$ (since $R^g(X,Y)\vert_{\spf(1)}=0$).  

Put attention on $R^g(X,Y)\vert_{\g_{\bar1}}\ne-\varphi_{X,Y}$, except for the symplectic triple system of   symplectic type (corresponding to the sphere, well known for being of constant curvature, or to the projective space, locally undistinguishable). Indeed, for $x,y,z\in T$   and $  a,b\in V$, then   
$$
R^g(a\otimes x,a\otimes y)(b\otimes z)=-\frac12\langle a,b\rangle a\otimes\big((x,z)y-(y,z)x\big)=-\varphi_{a\otimes x,a\otimes y}(b\otimes z),
$$
taking into account Eq.~\eqref{le_kappa}.
But, for $e_1,e_2\in V$ with $\langle e_1,e_2\rangle=1$, then $\gamma_{e_1,e_1}(e_2)=2e_1$, 
$\gamma_{e_2,e_1}(e_1)=-e_1$ and 
$$
R^g(e_1\otimes x,e_2\otimes y)(e_1\otimes z)=e_1\otimes \big(
(x,z)y+\frac12(y,z)x-[x,y,z]\big),
$$
which will coincide with
$$-\varphi_{e_1\otimes x,e_2\otimes y}(e_1\otimes z)=g({e_2\otimes y,e_1\otimes z})(e_1\otimes x)=-\frac12e_1\otimes (y,z)x,
$$
if and only if the identity
$[x,y,z]=(x,z)y+ (y,z)x$ holds in $T$. Note that the identity is false for a symplectic triple system not of symplectic type. It is interesting to remark that just this difference, $[x,y,z]-((x,z)y+ (y,z)x)$, measures how far is a 3-Sasakian manifold of being of constant curvature. 
\end{remark}

\subsection{ The distinguished connection}\label{subsec_dis}

\begin{lemma}\label{le_alfaS}
If $\xi,\xi'\in  \spf(1)$, $X,Y\in \g_{\bar1}$, then
$$
\begin{array}{ll}
\alpha^\mathcal{S}(\xi,\xi')=0,&\alpha^\mathcal{S}(X,\xi_i)=0,\\
\alpha^\mathcal{S}(\xi_i,X)=-\varphi_i(X),\qquad\qquad&\alpha^\mathcal{S}(X,Y)=0.
\end{array}
$$
\end{lemma}

\begin{proof}
Note that the fact $\alpha^\mathcal{S}(.,\xi)=0$ is the condition required for the choice of the affine connection in \cite[Theorem~5.6]{nues3Sas}. Anyway, it is easy to check it directly. Indeed, as $\alpha^\mathcal{S}-\alpha^g=2\alpha_o+ \sum_{r=1}^3 \alpha_{rr}$ is alternating, so   Eq.~\eqref{eq_torsiones} gives
$$
\begin{array}{ll}
(\alpha^\mathcal{S}-\alpha^g)(\xi,\xi')= [\xi,\xi']-\frac32[\xi,\xi'],&(\alpha^\mathcal{S}-\alpha^g)(X,\xi_i)=0+\varphi_i(X)=[\xi_i,X],\\
(\alpha^\mathcal{S}-\alpha^g)(\xi_i,X)=0-\varphi_i(X),\qquad\qquad&(\alpha^\mathcal{S}-\alpha^g)(X,Y)=0+\sum_r\Phi_r(X,Y)\xi_r=-\frac12[X,Y]_\mm;
\end{array}
$$
and we finish by taking into account Eq.~\eqref{eq_alfadeLevi}.
\end{proof}

\begin{proposition}
After complexifying, the curvature operators become
$$
\begin{array}{l}
R^\mathcal{S}(\xi,\xi')(\xi''+c\otimes z)=2[\xi,\xi'](c)\otimes z,\\
R^\mathcal{S}(a\otimes x, \xi)=0,\\
R^\mathcal{S}(a\otimes x, b\otimes y)(\xi+c\otimes z)=\gamma_{a,b}(c)\otimes (x,y)z-\langle a,b\rangle c\otimes [x,y,z],
\end{array}
$$
for any $\xi,\xi',\xi''\in  \spf(V,\langle.,.\rangle)$, $a,b,c\in V$, $x,y,z\in T$.
\end{proposition}

\begin{proof}
By recalling
$R^{\mathcal{S}}(X,Y)=[\alpha^\mathcal{S}_X,\alpha^\mathcal{S}_Y]-\alpha^\mathcal{S}_{[X,Y]_{\mm}}-\ad[X,Y]_{\hh}$, and $\alpha^\mathcal{S}_{a\otimes x}=0$, then
$$
R^{\mathcal{S}}(a\otimes x,\xi)=[0,\alpha^\mathcal{S}_\xi]-\alpha^\mathcal{S}_{-\xi(a)\otimes x}=0,
$$
and
$$
R^{\mathcal{S}}(a\otimes x,b\otimes y)=0-\alpha^\mathcal{S}_{(x,y)\gamma_{a,b}}-\ad{\langle a,b\rangle d_{x,y}}=-(x,y)\alpha^\mathcal{S}_{\gamma_{a,b}}-\langle a,b\rangle \ad{ d_{x,y}}.
$$
Finally we get  $[\alpha^\mathcal{S}_{\xi},\alpha^\mathcal{S}_{\xi'}]=-\alpha^\mathcal{S}_{[\xi,\xi']}$, because both maps are zero in the vertical part and $[\alpha^\mathcal{S}_{\xi},\alpha^\mathcal{S}_{\xi'}]\vert_{V\otimes T}=[\ad\xi,\ad\xi']=\ad[\xi,\xi']$. Thus $R^{\mathcal{S}}(\xi,\xi')=-2\alpha^\mathcal{S}_{[\xi,\xi']}$.
\end{proof}

\begin{theorem}
The complexification of the holonomy Lie algebra of the distinguished affine connection  is
$$
\begin{array}{rcl}
\hol(\nabla^\mathcal{S})^\CC&=
&\{f\in \gl(\mm^\CC):\exists \xi\in  \spf(V,\langle.,.\rangle)\textrm{ with }f\vert_{\spf(V,\langle.,.\rangle)}=0, f\vert_{V\otimes T}=\xi\otimes \id_T\}\\
&&\oplus\{f\in \gl(\mm^\CC):\exists d\in  \inder(T)\textrm{ with }f\vert_{\spf(V,\langle.,.\rangle)}=0, f\vert_{V\otimes T}=\id_V\otimes  d\}\\
&\cong& \slf_2(\CC)\oplus\inder(T);
\end{array}
$$ 
so that
$$\hol(\nabla^\mathcal{S})\cong \suf(2)\oplus\hh.$$
\end{theorem}

\begin{proof}
The linear span of the curvature operators  
$R^{\mathcal{S}}(\xi,\xi')=-2\alpha^\mathcal{S}_{[\xi,\xi']}$ and $R^\mathcal{S}(a\otimes x, b\otimes y)=-(x,y)\alpha^\mathcal{S}_{\gamma_{a,b}}-\langle a,b\rangle\ad d_{x,y}\in\gl(\mm^\CC)$
is the vector space (which turns out to be also a  Lie algebra)
\begin{equation}\label{eq_parahol}
\{\alpha^\mathcal{S}_{\xi}:\xi\in \spf(V,\langle.,.\rangle)\}\oplus\{\ad d\vert_{\mm^\CC}:d\in  \inder(T)\}.
\end{equation}
The complexification of the holonomy algebra, $\hol(\nabla^\mathcal{S})^\CC$, becomes the smallest Lie algebra containing such set and closed with brackets with $\alpha^\mathcal{S}_{\xi}=-\ad\xi\circ \pi_{\g_{\bar1}}$ for all $\xi\in\spf(V,\langle.,.\rangle)$, since $\alpha^\mathcal{S}_{a\otimes x}=0$. As $\ad d_{x,y}$ acts on $T$, and $\ad\xi$ acts on $V$, they commute, so that
$$
[R^\mathcal{S}(a\otimes x, b\otimes y),\alpha^\mathcal{S}_{\xi}]=-(x,y)[\alpha^\mathcal{S}_{\gamma_{a,b}},\alpha^\mathcal{S}_{\xi}]=(x,y)\alpha^\mathcal{S}_{[\gamma_{a,b},\xi]}.
$$
Hence, the set in Eq.~\eqref{eq_parahol}  is closed for the required brackets and hence it coincides with the whole Lie algebra $\hol(\nabla^\mathcal{S})^\CC$.  
\end{proof}

If we compare $\hol(\nabla^\mathcal{S})$ with $\hol(\nabla^g)$, we see that this holonomy algebra is considerably smaller, what indicates that the distinguished connection is better adapted than $\nabla^g$ to the geometry of the 3-Sasakian manifolds, as expected.

 
 \subsection{ The canonical connection}\label{subsec_can}

\begin{lemma}
If $\xi,\xi'\in  \spf(1)$, $X,Y\in \g_{\bar1}$, then
$$
\begin{array}{ll}
\alpha^c(\xi,\xi')=-[\xi,\xi'],&\alpha^c(X,\xi_i)=0,\\
\alpha^c(\xi_i,X)=-\varphi_i(X),\qquad\qquad&\alpha^c(X,Y)=0.
\end{array}
$$
\end{lemma}
 
 \begin{proof}
Simply observe that $\alpha^c-\alpha^\mathcal{S}=-2\alpha^o$, so that by Eq.~\eqref{eq_torsiones},
$$
\begin{array}{ll}
\alpha^c(\xi,\xi')=\alpha^\mathcal{S}( \xi,\xi')-[\xi,\xi'],\qquad\qquad
& \alpha^c(X,\xi)=\alpha^\mathcal{S}(X,\xi),\\
\alpha^c( \xi,X)=\alpha^\mathcal{S}( \xi,X),&\alpha^c(X,Y)=\alpha^\mathcal{S}(X,Y),
\end{array}
$$
and by  Lemma~\ref{le_alfaS} the result follows.
\end{proof}

\begin{proposition}
After complexifying,   the curvature operators become
$$
\begin{array}{l}
R^c(\xi,\xi')=2\ad [\xi,\xi']\vert_{\mm^\CC},\\
R^c(a\otimes x, \xi)=0,\\
R^c(a\otimes x, b\otimes y)=\ad\big((x,y)\gamma_{a,b}-\langle a,b\rangle d_{x,y}\big)\vert_{\mm^\CC},
\end{array}
$$
for any $\xi,\xi'\in  \spf(V,\langle.,.\rangle)$, $a,b\in V$, $x,y\in T$.
\end{proposition}

\begin{proof}
First, as $[\xi,\xi']_{\mm}=[\xi,\xi']$ and $[\xi,\xi']_{\hh}=0$, then, by Eq.~\eqref{eq_R},
$R^{c}(\xi,\xi')=[\alpha^c_\xi,\alpha^c_{\xi'}]-\alpha^c_{[\xi,\xi']}=[-\ad\xi,-\ad\xi']+\ad[\xi,\xi']=2\ad[\xi,\xi']$.
Second, $\alpha^c_{a\otimes x}=0=\alpha^c_{\xi(a)\otimes x}$, what implies $R^c(a\otimes x, \xi)=0$.
And third,
$R^c(a\otimes x, b\otimes y)=0-(x,y)\alpha^c_{\gamma_{a,b}}-\langle a,b\rangle \ad d_{x,y}.$
\end{proof}

\begin{remark}
Observe some analogies between the curvature operators related to the distinguished and canonical connections: $R^\mathcal S(X,Y)\vert_{\g_{\bar1}}=R^c(X,Y)\vert_{\g_{\bar1}}$ for all $X,Y\in\mm$, although they do not coincide in the vertical part, since $R^\mathcal S(X,Y)\vert_{\spf(1)}=0$ (the same than   happens after the complexification).

Note also that the first Bianchi identity is of course not longer true for connections with torsion, because, by using Remark~\ref{re_Lts},  
\begin{center}$\sum_{\tiny
 \begin{array}{l}\textrm{cyclic} \vspace{-4pt}\\a,b,c\vspace{-4pt}\\x,y,z\end{array}} R^\mathcal S(a\otimes x, b\otimes y,c\otimes z)=2\big(\gamma_{a,b}(c)\otimes (x,y)z+\gamma_{b,c}(a)\otimes (y,z)x+\gamma_{c,a}(b)\otimes (z,x)y\big),
 $\end{center}
 which is obviously not always zero.
\end{remark}

\begin{theorem}
The complexification of the holonomy Lie algebra  of the canonical affine connection is
$$
\hol(\nabla^c)^\CC=\ad(\g(T)_{\bar0})\vert_{\mm^\CC}\cong\g(T)_{\bar0}=\spf(V,\langle.,.\rangle)\oplus \inder(T);
$$ 
and hence
$$
\hol(\nabla^c)\cong\suf(2)\oplus\hh.
$$
\end{theorem}

\begin{proof}
In this case, the maps $\alpha^c_\xi=-\ad\xi\vert_{\mm^\CC}\in R^c(\spf(V,\langle.,.\rangle),\spf(V,\langle.,.\rangle))$ are obviously included in $\hol(\nabla^c)^\CC$, by the above proposition. Hence $
R^c(a\otimes x, b\otimes y)-(x,y)\ad \gamma_{a,b}=-\langle a,b\rangle\ad d_{x,y} \vert_{\mm^\CC}$ belongs to the holonomy algebra too and the whole $\ad(\g(T)_{\bar0})$ is a subalgebra of $\hol(\nabla^c)^\CC$. As this subalgebra is already closed for the bracket with the maps $\alpha^c_\xi=-\ad\xi$, then it has to coincide with $\hol(\nabla^c)^\CC$.
\end{proof}

Observe that $\hol(\nabla^\mathcal{S})\cong\hol(\nabla^c)$, so that both holonomy algebras are isomorphic Lie subalgebras of $\sof(n)$ but different.   The point is that $\nabla^c$ parallelizes the torsion while $\nabla^\mathcal{S}$ parallelizes the Reeb vector fields.

The holonomy algebras (and groups) are semisimple in the two considered cases,   distinguished and canonical, whenever $T$ is not of special type. But if $M=\frac{\SU(m)}{S(\mathrm{U}(m-2)\times \mathrm{U}(1))}$, then $\hh\cong\mathfrak{u}(m-2)$  and hence the holonomy algebras have a one-dimensional center. 

\begin{remark}
Note that the invariant affine connection related to the bilinear map $\alpha=0$ gives a holonomy algebra even smaller, since it can be proved to be isomorphic to $\hh$. But the geometric properties are not very good since, for instance, it does not have skew-symmetric torsion.
\end{remark}

\begin{remark}
Another invariant metric affine connection with skew-torsion emphasized in \cite{nues3Sas} is the characteristic connection $\nabla_{\tau}^{ch}$ related to any of the Sasakian structures $\{\xi_{\tau} ,\eta_{\tau}, \varphi_{\tau}\}$ (for some $\tau\in\mathbb S^2$). It is the only one satisfying   $\nabla^{ch}_{\tau} \xi_{\tau}=0, \nabla^{ch} _{\tau}\eta_{\tau}=0$ and $\nabla^{ch}_{\tau} \varphi_{\tau}=0$, and moreover, 
 it parallelizes the (skew-symmetric) torsion $T_{\tau}^{ch}$. In spite of that, we have not included  here the study of its holonomy algebra because the related curvature operators turn out to be very intrincate so that its holonomy algebra is not enough small.
\end{remark}

Here we enclose a table for comparing the dimensions of the obtained holonomy algebras for any (simply-connected)  3-Sasakian homogeneous manifold. \medskip 

\begin{center}
\hspace{-1.75cm}
\begin{tabular}{ c|| c | c| c | c |c| c | c | c |}
&$\tiny\frac{\Sp(n+1)}{\Sp(n)}$&$
\frac{\SU(m)}{S(\mathrm{U}(m-2)\times \mathrm{U}(1))}$&$
\frac{\SO(k )}{\SO(k-4)\times \Sp(1)}$&$
\frac{G_2 }{\Sp(1) }$&$  
\frac{F_4 }{\Sp(3) }$&$
\frac{E_6 }{\SU(6) }$&$
\frac{E_7 }{\mathrm{Spin}(12) }$
&$ \frac{ E_8}{ E_7}$\vspace{3pt}\\
{$n$}& &\scriptsize{$m-2$}&\scriptsize{$k-4$}&\scriptsize{2}&\scriptsize{7}&\scriptsize{10}&\scriptsize{16}&\scriptsize{28}\\\hline\hline
$\dim\hol(\nabla^g)$&$8n^2+10n+3$&$8n^2+10n+3$&$8n^2+10n+3$&55&465&903&2211&6555\\\hline
$\dim\hol(\nabla^\mathcal{S})$&$2n^2+n+3$&$n^2+3$&$\tiny\frac{n^2-n}2+6$&$6$&$24$&38&69&136\\\hline
$\dim\hol(\nabla^c)$&$2n^2+n+3$&$n^2+3$&$\tiny\frac{n^2-n}2+6$&$6$&$24$&38&69&136\\\hline\hline
\end{tabular}
\end{center}
\medskip 

Other possible comparisons, for instance the study of the Ricci tensor, are directly extracted from \cite[Proposition~5.2 and Corollary~5.4]{nues3Sas}.
The three Ricci tensors vanish in $\spf(1)\times\g_{\bar1}$, as the metric does, and 
$$
\begin{array}{ll}
\textrm{Ric}^g\vert_{\spf(1)\times \spf(1)}=(4n+2)g,\qquad\qquad&\textrm{Ric}^g\vert_{ \g_{\bar1}\times\g_{\bar1}}=(4n+2)g,\\
\textrm{Ric}^\mathcal{S}\vert_{\spf(1)\times \spf(1)}=0,\qquad&\textrm{Ric}^\mathcal{S}\vert_{ \g_{\bar1}\times\g_{\bar1}}=(4n-4)g,\\
\textrm{Ric}^c\vert_{\spf(1)\times \spf(1)}=-16g,\qquad&\textrm{Ric}^c\vert_{ \g_{\bar1}\times\g_{\bar1}}=(4n-4)g.
\end{array}
$$
In particular, $\nabla^\mathcal{S}$ and $\nabla^c$ are $\mathcal{S}$-Einstein invariant affine connections in the sense of \cite{nues3Sas}. 
Moreover, $\nabla^\mathcal S$ is Ricci-flat (but not flat) if $\dim M=7$, that is, if $M$ is either the sphere $\mathbb{S}^7$, or the projective space
$\RR P^7$, or the Aloff-Wallach space $\mathfrak{W}^{7}_{1,1}=\SU(3)/\mathrm U(1)$.

If we recall too that the scalar curvature is given by $(4n+2)(4n+3)-\frac32(a-\tr(B))^2-3n\|B\|^2$ for $B=(b_{rs})$ the matrix of the coefficients in Eq.~\eqref{eq_skewtor}, then, for $\dim M=7$,
$$
s^g=42;\quad
s^\mathcal{S}=0;\quad
s^c=-48;
$$
and, for arbitrary dimension, 
$$
s^g=(4n+2)(4n+3);\quad
s^\mathcal{S}=16n(n-1);\quad
s^c=16(n^2-n-3).
$$
Again it is stricking that these results only depend on the dimension of $M$, but not on the concrete  
homogeneous 3-Sasakian manifold we are working with.

\end{document}